\documentclass[12pt]{amsart}

\usepackage{amssymb}
\usepackage{graphicx}
\usepackage{accents}
\usepackage[english,german]{babel}
\usepackage{color}
\usepackage{cancel}
\usepackage{url}

\newtheorem{theorem}{Theorem}[section]
\newtheorem{lemma}[theorem]{Lemma}
\newtheorem{Prop}[theorem]{Proposition}

\theoremstyle{definition}
\newtheorem{definition}[theorem]{Definition}

\theoremstyle{remark}
\newtheorem{remark}[theorem]{Remark}

\numberwithin{equation}{section}

\makeatletter
\newcommand*{\rom}[1]{\text{\expandafter\@slowromancap\romannumeral #1@}}
\makeatother
\newcommand{\R}{\mathbb{R}}
\newcommand{\N}{\mathbb{N}}
\newcommand{\Z}{\mathbb{Z}}

\newcommand{\Ric}{\operatorname{Ric}}

\newcommand{\Scal}{\operatorname{R}}

\newcommand{\grad}{\operatorname{grad}}
\newcommand{\diver}{\operatorname{div}}

\newcommand{\mylap}[1]{{}^{#1}\!\triangle}
\newcommand{\my}[2]{{}^{#1}{#2}}

\newcommand{\free}[1]{\accentset{\,\circ}{#1}}

\newcommand{\spacetime}{(\R\times\slice,-N^{2}\,dt^{2}+g)}

\newcommand{\photo}{P^{3}}

\newcommand{\slice}{M^{3}}

\newcommand{\sphere}{\mathbb{S}^{2}}
\newcommand{\surf}{\Sigma^{2}}
\newcommand{\RN}{Reiss\-ner-Nord\-str\"om }

\newcommand{\spacet}{\R\times\slice}

\newcommand{\ES}{(\slice,g,N,\Phi)}

\newcommand\beq{\begin{equation}}
\newcommand\eeq{\end{equation}}
\newcommand\ben{\begin{enumerate}}
\newcommand\een{\end{enumerate}}
\newcommand\bit{\begin{itemize}}
\newcommand\eit{\end{itemize}}

\begin{document}
\selectlanguage{english}
\title[Uniqueness of photon spheres]{Uniqueness of photon spheres\\ in electro-vacuum spacetimes}
\author{Carla Cederbaum}
\address{Mathematics Department, Universit\"at T\"ubingen, Germany}
\email{cederbaum@math.uni-tuebingen.de}
\thanks{The first author is indebted to the Baden-W\"urttemberg Stiftung for the financial support of this research project by the Eliteprogramme for Postdocs.}

\author{Gregory J. Galloway}
\address{Mathematics Department, University of Miami, USA}
\email{galloway@math.miami.edu}
\thanks{The second author was partially supported by NSF grant DMS-1313724.}

\date{}

\begin{abstract}
In a recent paper \cite{cedergal}, the authors established the uni\-queness of photon spheres in static vacuum asymptotically flat spacetimes by adapting Bunting and Masood-ul-Alam's proof of static vacuum black hole uniqueness. Here, we establish uniqueness of suitably defined sub-extremal photon spheres in static electro-vacuum asymptotically flat spacetimes by adapting the argument of Masood-ul-Alam~\cite{MuA}.

As a consequence of our result, we can rule out the existence of electrostatic configurations involving multiple ``very compact'' electrically charged bodies and sub-extremal black holes.\end{abstract}

\maketitle

\section{Introduction}\label{sec:intro}
The static, spherically symmetric, electrically charged \RN black hole spacetime of mass $M>0$ and charge $Q\in\R$ can be represented as
\begin{align}
({\mathfrak{L}}^{4}:=\R\times(\R^{3}\setminus \overline{B_{R}(0)}),{\mathfrak{g}}),
\end{align}
where $R:=M+\sqrt{M^{2}-Q^{2}}$ if $Q^{2}\leq M^{2}$ and $R:=0$ otherwise, and the Lorentzian metric ${\mathfrak{g}}$, the lapse $N$, and the electric potential $\Phi$ are given by
\begin{align}\label{RNmetric}
{\mathfrak{g}}&=-{N}^{2}dt^{2}+{N}^{-2}dr^{2}+r^{2}\,\Omega,\quad 
N=\sqrt{1-\frac{2M}{r}+\frac{Q^{2}}{r^2}},
\quad\text{and } \Phi=\frac{Q}{r},
\end{align}
with $\Omega$ denoting the canonical metric on $\mathbb{S}^{2}$ and $r$ the radial coordinate on $\R^{3}$. 

The \RN spacetime is called sub-extremal if $Q^{2}<M^{2}$, extremal if $Q^{2}=M^{2}$ and otherwise super-extremal. The sub-extremal \RN spacetime is well-known to possess a black hole horizon at 
$r=M+\sqrt{M^{2}-Q^{2}}$. 
The super-extremal \RN spacetimes do not contain black holes but rather feature naked singularities. 
Just like the Schwarzschild spacetime, the sub-extremal and extremal \RN spacetimes each possess a unique \emph{photon sphere} $\photo_{+}$, while the super-extremal \RN spacetimes with $8Q^{2}<9M^{2}$ each possess precisely two photon spheres $\photo_{\pm}$ and the \RN spacetimes with $8Q^{2}=9M^{2}$ each possess a unique photon sphere $\photo_{+}=\photo_{-}$, with
\begin{align}\label{def:RNphoto}
\photo_{\pm}:=\R\times\sphere_{\frac{3M}{2}\pm\frac{1}{2}\sqrt{9M^{2}-8Q^{2}}}=\left\lbrace r=\frac{3M}{2}\pm\frac{1}{2}\sqrt{9M^{2}-8Q^{2}}\right\rbrace.
\end{align}
The \RN spacetimes with $8Q^2>9M^2$ do not possess photon spheres.

Here, a photon sphere is understood to be a timelike warped cylinder submanifold $\photo=\R\times\sphere_{r}$ such that any null geodesic of $({\mathfrak{L}}^{4},{\mathfrak{g}})$ that is initially tangent to $\photo$ remains tangent to it. The \RN photon spheres thus model (embedded submanifolds ruled by) photons spiraling around the central black hole or naked singularity ``at a fixed distance''. 

Photon spheres and the notion of trapped null geodesics in general are crucially relevant for questions of dynamical stability in the context of the Einstein equations. Moreover, photon spheres are related to the existence of relativistic images in the context of gravitational lensing. See \cite{CederPhoto,yazadjiev} and the references cited therein for more information on photon spheres.

To the best of our knowledge, it is mostly unknown whether more general spacetimes can possess (generalized) photon spheres, see p.\;838 of \cite{CVE}. Recently, the first author gave a geometric definition of photon spheres in static spacetimes and proved uniqueness of photon spheres in $3+1$-dimensional asymptotically flat static vacuum spacetimes under the assumption that the lapse function of the spacetime regularly foliates the region exterior to the photon sphere \cite{CederPhoto}. This definition and the resultant uniqueness result have since been adopted and generalized from vacuum to (non-extremal) electro-vacuum by Yazadjiev and Lazov \cite{LY}. They, too, assume that the lapse function of the spacetime regularly foliates the exterior region of the photon sphere. In particular, both results assume a priori that there is only one (connected) photon~sphere\footnote{To be precise, both results implicitly assume that the spacetime possesses one (connected) photon sphere arising as its inner boundary. Neither of the results additionally assumes that there cannot be additional photon spheres in the interior of the spacetime. This is particularly relevant in the super-extremal charged case.}.

Adopting the definition of photon spheres in electro-vacuum \cite{LY}, we will prove photon sphere uniqueness for $3+1$-dimensional asymptotically flat, static, electro-vacuum spacetimes without assuming that the lapse function regularly foliates the spacetime. In particular, we allow a priori the possibility of multiple photon spheres. This generalizes our results for vacuum photon spheres and black holes \cite{cedergal}.

To accomplish this, we make use of the rigidity case of the Riemannian positive mass theorem (under the weaker regularity assumed in \cite{Bartnik}, see also \cite{LeeLefloch} and references cited therein), in a manner similar to the proof of black hole uniqueness in electro-vacuum due to Masood-ul-Alam~\cite{MuA} and to our proof of photon sphere uniqueness in vacuum~\cite{cedergal}. Our strategy of proof only allows us to treat what we will call \emph{sub-extremal} photon spheres, see Definition \ref{def:sub-extremal}. For a motivation of this definition and a (partial) discussion of the extremal and super-extremal situation see Appendix \ref{app:sub-extremal}.

This paper is organized as follows. In Section \ref{sec:setup and definition}, we will recall the definition and a few properties of photon spheres in static, electro-vacuum, asymptotically flat spacetimes from \cite{CederPhoto,yazadjiev}. In Section~\ref{sec:proof}, we will prove that the only such spacetimes admitting sub-extremal photon spheres are the sub-extremal \RN spacetimes:
{\renewcommand{\thetheorem}{\ref{thm:main}}
\begin{theorem}
Let $\ES$ be an electrostatic system which is electro-vacuum, asymptotic to Reissner-Nordstr\"om, and possesses a (possibly disconnected) sub-extre\-mal photon sphere $(\photo,p)\hookrightarrow\spacetime$, arising as the inner boundary of $\spacet$. Let $M$ denote the ADM-mass and $Q$ the total charge of $\spacetime$. Then $M>0$, $Q^{2}<M^{2}$, and $\spacetime$ is isometric to the region $\lbrace r\geq \frac{3M}{2}+\frac{1}{2}\sqrt{9M^{2}-8Q^{2}}\rbrace$ exterior to the photon sphere $\lbrace r=\frac{3M}{2}+\frac{1}{2}\sqrt{9M^{2}-8Q^{2}}\rbrace$ in the \RN spacetime of mass~$M$ and charge $Q$. In particular, $(\photo,p)$ is connected and a cylinder over a topological sphere.
\end{theorem}
\addtocounter{theorem}{-1}}

\begin{remark} 
As in \cite{CederPhoto,LY,cedergal}, one does not need to assume a priori that the mass $M$ of the spacetime is positive or that the total mass $M$ and charge $Q$ satisfy the sub-extremality condition $Q^{2}<M^{2}$;  this is a consequence of the theorem.  In particular, the existence of sub-extremal photon spheres in static, electro-vacuum, asymptotically flat spacetimes of non-positive mass is ruled out.
\end{remark}

\begin{remark}\label{rem:multiple}
In addition to the photon sphere inner boundary, our arguments allow for the presence of non-degenerate (Killing) horizons as additional components of the boundary of the spacetime.  These would be treated just as in the original argument by Masood-ul-Alam \cite{MuA}. For simplicity, we have assumed no such horizon boundaries.
\end{remark}

In Section \ref{sec:nbody}, we explain how Theorem \ref{thm:main} can be applied to the so called electrostatic $n$-body problem in General Relativity, yielding the following corollary:

{\renewcommand{\thetheorem}{\ref{thm:nbody}}
\begin{theorem}[No static configuration of $n$ ``very compact'' electrically charged bodies and black holes]
There are no electrostatic equilibrium configurations of $k\in~\!\!\N$ possibly electrically charged bodies and $n\in\N$ possibly electrically charged non-degenerate black holes with $k+n>1$ in which each body is surrounded by its own sub-extremal photon sphere.
\end{theorem}
\addtocounter{theorem}{-1}}

\section{Setup and definitions}\label{sec:setup and definition}
Static electro-vacuum spacetimes model exterior regions of static configurations of electrically charged stars and/or black holes. They can be mathematically described as \emph{electrostatic systems} $\ES$, where $(\slice,g)$ is a smooth Riemannian manifold, geodesically complete up to its inner boundary, $N:\slice\to\R^+$ is the (smooth) \emph{lapse} function, $\Phi:\slice\to\R$ denotes the (smooth) \emph{electric potential}, and the following symmetry reduced Einstein-Maxwell or \emph{electro-vacuum} equations hold:
\begin{align}\label{SMEelectrovac1}
\diver\left(\frac{\grad \Phi}{N}\right)&=0\\\label{SMEelectrovac2}
\triangle N&=\frac{\vert d\Phi\vert^2}{N}\\\label{SMEelectrovac3}
N\,{\Ric}&={\nabla}^2 N+\frac{\vert d\Phi\vert^2\,g-2\,d\Phi\otimes d\Phi}{N},
\end{align}
on $\slice$, where $\diver$, $\grad$, $\triangle$, $\nabla^{2}$, and $\Ric$ denote the covariant divergence, gradient, Laplacian, Hessian, and Ricci curvature of the metric $g$, respectively. 

Such electrostatic systems arise as (canonical) time-slices of (standard) static spacetimes
\begin{align}\label{ST}
\spacetime
\end{align}
which accordingly solve the Einstein-Maxwell equations. An electrostatic system is called \emph{asymptotic to Reissner-Nordstr\"om} if the manifold~$\slice$ is diffeomorphic to the union of a compact set and an open \emph{end} $E^3$ which is diffeomorphic to $\R^3\setminus \overline{B}$ with $B$ the open unit ball in $\R^3$. Furthermore, we require that, in the end $E^3$, the lapse function $N$, the electric potential $\Phi$, the metric $g$, and the coordinates $(x^{i}):E^3\to\R^3\setminus \overline{B}$ combine such that there are constants $M,Q\in\R$ for which we find the decay
\begin{align}\label{AF}
g_{ij}-\left(1+\frac{2M}{r}\right)\delta_{ij}&\in W^{k,q}_{-\tau}(E)\\\label{NAF}
N-\left(1-\frac{M}{r}\right)\phantom{g_{ij}}&\in W^{k,q}_{-\tau}(E)\\\label{PhiAF}
\Phi\;-\;\;\;\left(\frac{Q}{r}\right)\phantom{g_{ij}-}&\in W^{k,q}_{-\tau}(E)
\end{align}
in the weighted Sobolev space $W^{k,q}_{-\tau}(E)$ for some $\tau>3/2$, $\tau\notin\Z$, $k\geq2$, $q>4$, where we are using Bartnik's Sobolev space conventions \cite{Bartnik}. Here, $\delta$ denotes the Euclidean metric on~$\R^3$ and  $r:=\sqrt{(x^1)^{2}+(x^2)^{2}+(x^3)^{2}}$ denotes the radial coordinate corresponding to the coordinates $(x^{i})$ on $E^3$. Combining \eqref{SMEelectrovac1} through \eqref{SMEelectrovac3}, one finds
\begin{align}\label{eq:scal}
\Scal&=\frac{2\,\vert d\Phi\vert^2}{N^2}
\end{align}
on $\slice$, where $\Scal$ denotes the scalar curvature of $g$.

It is convenient for us, as assumed above, to take $M^3$ to be a manifold with boundary. 
However, it is to be understood that the electrostatic system extends to a slightly larger electrostatic system containing $\partial M$ and satisfying $N>0$.

\subsection{Definition of photon spheres}
We adopt the definition of photon spheres from \cite{CederPhoto} with the electrostatic extension suggested by \cite{LY}:
\begin{definition}[Photon sphere]\label{def:photo}
Let $\ES$ be an electrostatic system with associated static spacetime $\spacetime$ according to \eqref{ST}.
A timelike embedded hypersurface $(\photo,p) \hookrightarrow\spacetime$ is called a \emph{(generalized) photon sphere} if the embedding $(\photo,p) \hookrightarrow\spacetime$ is umbilic and if the lapse function $N$ and the electric potential $\Phi$ are constant on each connected component of~$\photo$.
\end{definition}

\begin{remark}
Note that, a priori, the definition of a \emph{(generalized) photon sphere}  neither requires the photon sphere to be connected nor to have the topology of a cylinder over a sphere. This enables us to treat multiple photon spheres at once. In fact, as a consequence of Proposition \ref{prop} and Lemma \ref{untrapped} presented below, we will see that the topology of each component of a (generalized) photon sphere  is that of a cylinder over a sphere.
\end{remark}

Our sign convention is such that the second fundamental form $h$ of an isometrically embedded $2$-surface $(\surf,\sigma)\hookrightarrow (\slice,g)$ with respect to the outward unit normal vector field $\nu$ is chosen such that $h(X,Y):=g(\my{g}{\nabla}_{X} \nu,Y)$ for all $X,Y\in\Gamma(\surf)$. In other words, the round spheres in flat space have positive mean curvature with respect to the outward unit normal according to our convention.

\begin{remark}\label{rem:vac}
A photon sphere as defined here and in \cite{LY} is a generalization of photon spheres as defined in \cite{CederPhoto}, where no assumptions are made about fields other than $N$ and the metric/second fundamental form on the timelike hypersurface $(\photo,p)$.
\end{remark}

\subsection{Properties of photon spheres in static electro-vacuum}\label{subsec:props}
Let $(\photo,p)$ be a photon sphere which arises as the inner boundary of a static spacetime 
$\spacetime$, and let $N_i:=N\vert_{\photo_{i}}$ denote the (constant) value of $N$ on the connected component 
$(\photo_i,p_i)$ of $(\photo,p)$ for all $i=1,\dots,I$. As the photon sphere $(\photo,p)$ arises as the inner boundary of a standard static spacetime as in \eqref{ST}, each component $(\photo_{i},p_{i})$ must be a warped cylinder $(\photo_i,p_i)=(\R\times\surf_i,-N_i^2dt^2+\sigma_i)$,
where $\surf_i$ is the (necessarily compact) intersection of the photon sphere component $\photo_i$ and the time slice $\slice$ and $\sigma_i$ is the (time-independent) induced metric on $\surf_i$. We set 
$\surf:=\cup_{i=1}^I\surf_i$ and let $\sigma$ be the metric on $\surf$ that coincides with $\sigma_i$ on $\surf_i$. 

Then, photon spheres have the following local\footnote{which are derived without appealing to the asymptotic decay at infinity.} properties:
\begin{Prop}[Cederbaum \cite{CederPhoto}, Yazadjiev-Lazov \cite{LY}]\label{prop}
Let $\ES$ be an electrostatic system solving the electro-vacuum equations \eqref{SMEelectrovac1}-\eqref{SMEelectrovac3}, and let $(\photo,p)\hookrightarrow\spacetime$ be a (generalized) photon sphere arising as the inner boundary of the associated spacetime $\spacetime$. We write 
\begin{align}
\left(\photo,p\right)=\left(\R\times\surf,-N^{2}dt^2+\sigma\right)= \bigcup_{i=1}^I\left(\R\times\surf_i,-N_i^2dt^2+\sigma_i\right),
\end{align}
where each $\photo_{i}=\R\times\surf_{i}$ is a connected component of $\photo$. Then the embedding $(\surf,\sigma)\hookrightarrow(\slice,g)$ is totally umbilic with constant mean curvature $H_i$ on the component $\surf_i$. The normal derivatives of the lapse function $N$ and the electric potential $\Phi$ in direction of the outward unit normal $\nu$ to $\surf$, $\nu(N)$ and $\nu(\Phi)$, respectively, are also constant on every component $(\surf_i,\sigma_i)$, and we set $\nu(N)_i := \nu(N)\vert_{\surf_{i}}$, $\nu(\Phi)_{i}:=\nu(\Phi)\vert_{\surf_{i}}$.  Moreover, the scalar curvature of the component $(\surf_i,\sigma_i)$, $\my{\sigma_i}{\Scal}$, is a non-negative constant which can be computed from the other constants via
\begin{align}\label{eq:scalprop}
\my{\sigma}{\Scal}_i&=\frac{3}{2}H_{i}^{2}+2\left(\frac{(\nu(\Phi)_{i})}{N_{i}}\right)^{2}.
\end{align}
For each $i\in\lbrace1,\dots,I\rbrace$, either $H_{i}=\nu(\Phi)_{i}=0$ and $\surf_{i}$ is a totally geodesic flat torus or $\surf_{i}$ is an intrinsically and extrinsically round CMC sphere for which the above constants are related via
\begin{align}\label{eq:prop}
N_iH_i^{2}&=2\nu(N)_iH_{i}.
\end{align}
\end{Prop}

\begin{remark}
We have assumed the existence of a globally defined electric potential $\Phi$, which will usually be asserted in a static electro-vacuum spacetime $\spacetime$ by using that $\slice$ is simply connected.  
In fact, one can show in our setting that $\slice$ is necessarily simply connected; see Remark \ref{ends}.
\end{remark}

The following lemma from \cite{cedergal} also applies in the electro-vacuum case as the electro-magnetic energy-momentum tensor satisfies the null energy condition (NEC), see also Remark \ref{rem:vac}:
\begin{lemma}[Cederbaum-Galloway {\cite[Lemma 2.6]{cedergal}}, Galloway-Miao {\cite[Theorem~3.1]{GalMiao}}]\label{untrapped}
Let $\spacetime$ be a static, asymptotically flat spacetime satisfying the NEC
and having a photon sphere inner boundary $(\photo,p)\hookrightarrow\spacetime$.  Then each component $\surf_{i}$ of $\surf:=\photo\cap\slice$ has (constant) positive mean curvature, $H_i >0$.
\end{lemma}

\begin{remark}
Lemma \ref{untrapped} rules out the torus case in Proposition \ref{prop} and thus ensures that each photon sphere component is diffeomorphic to a cylinder over a sphere. Moreover, it ensures via \eqref{eq:prop} that not only $H_{i}>0$ but also $\nu(N)_{i}>0$ on each component $\surf_{i}$. We can thus express the scalar curvature $\my{\sigma}{\Scal}_i$ in terms of the area radius
\begin{align}\label{eq:arearadius}
r_{i}&:=\sqrt{\frac{\vert\surf_{i}\vert_{\sigma_{i}}}{4\pi}}
\end{align}
and rewrite \eqref{eq:scalprop} as
\begin{align}\label{eq:constraint}
\frac{4}{3}&=r_{i}^{2} H_{i}^{2}+\frac{4}{3}\left(\frac{r_i (\nu(\Phi)_{i})}{N_{i}}\right)^{2}.
\end{align}
Moreover, Equation \eqref{eq:prop} simplifies to
 \begin{align}\label{eq:prop'}
N_iH_i&=2\nu(N)_i.
\end{align}
\end{remark}

\begin{remark}\label{ends}
In our definition of an electrostatic system being asymptotic to Reissner-Nordstr\"om, we have assumed just one asymptotic end.  If, however, we had allowed several asymptotic ends in the definition, the argument used to prove Lemma \ref{untrapped}, together with the assumed metric decay \eqref{AF},
could  be used to prove that there can be only one asymptotic end.  Applying this to the universal cover further implies that $\slice$ must be simply connected.\end{remark}

Before we proceed to prove Theorem \ref{thm:main}, we still need to define what we mean by a sub-extremal photon sphere. We will give an ad hoc definition here and provide some justification for this definition in Appendix \ref{app:sub-extremal}.

\begin{definition}\label{def:sub-extremal}
Let $\ES$ be an electrostatic system possessing a photon sphere $(\photo,p)\hookrightarrow\spacetime$. We call $\photo=\cup_{i=1}^{I}\photo_{i}$ \emph{sub-extremal} (or \emph{extremal} or \emph{super-extremal}) if the mean curvature $H_{i}$ of the components $\surf_{i}\hookrightarrow(\slice,g)$ and the area radius $r_{i}$ satisfy
\begin{align}\label{eq:subextremal}
H_{i}r_{i}&>1 \quad\text{ (or }H_{i}r_{i}=1\text{ or } H_{i}r_{i}<1)
\end{align}
for all $i=1,\dots,I$.
\end{definition}

\section{Proof of the main theorem}\label{sec:proof}
In this section we prove the `electrostatic photon sphere uniqueness theorem':
\begin{theorem}\label{thm:main}
Let $\ES$ be an electrostatic system which is electro-vacuum, asymptotic to Reissner-Nordstr\"om, and possesses a (possibly disconnected) sub-extre\-mal photon sphere $(\photo,p)\hookrightarrow\spacetime$, arising as the inner boundary of $\spacet$. Let $M$ denote the ADM-mass and $Q$ the total charge of $\spacetime$. Then $M>0$, $Q^{2}<M^{2}$, and $\spacetime$ is isometric to the region $\lbrace r\geq \frac{3M}{2}+\frac{1}{2}\sqrt{9M^{2}-8Q^{2}}\rbrace$ exterior to the photon sphere $\lbrace r=\frac{3M}{2}+\frac{1}{2}\sqrt{9M^{2}-8Q^{2}}\rbrace$ in the \RN spacetime of mass~$M$ and charge $Q$. In particular, $(\photo,p)$ is connected and a cylinder over a topological sphere.
\end{theorem}

\begin{proof}[Proof of Theorem \ref{thm:main}]
The main idea of our proof is as follows: In Step 1, we will define a new electro-vacuum electrostatic system $(\widetilde{M}^3,\widetilde{g},\widetilde{N},\widetilde{\Phi})$ which is asymptotic to \RN and has (Killing) horizon boundary by gluing in some carefully chosen pieces of (spatial) \RN manifolds of appropriately chosen masses and charges. More precisely, at each photon sphere base $\surf_{i}$, we will glue in a ``neck'' piece of a \RN manifold of mass $\mu_{i}>0$ and charge $q_{i}$, namely the cylindrical piece between its photon sphere and its horizon. This creates a new horizon boundary corresponding to each $\surf_{i}$. The manifold $\widetilde{M}^{3}$ itself is smooth while the metric $\widetilde{g}$, the lapse function $\widetilde{N}$, and the electric potential $\widetilde{\Phi}$ are smooth away from the gluing surfaces, and, as will be shown, $C^{1,1}$ across them. Also, away from the gluing surfaces, $(\widetilde{M}^3,\widetilde{g}\,)$ has non-negative scalar curvature.

Then, in a very short Step 2, we double the glued manifold over its minimal boundary\footnote{In case the original manifold $M^3$ had additional non-degenerate Killing horizon boundary components, $\widetilde{M}^3$ has additional boundary components with $N=0$ and thus $H=0$, $\free{h}=0$, $\my{\sigma\!}{R}\equiv\text{const.}$ The minimal boundary $\mathfrak{B}$ constructed here is thus of the same geometry and regularity as the spatial slices of the event horizon and the doubling can be carried out for both at once.} $\mathfrak{B}$ and assert that the resulting electrostatic system---which will also be called $(\widetilde{M}^3,\widetilde{g},\widetilde{N},\widetilde{\Phi})$---is in fact smooth across $\mathfrak{B}$. The resulting manifold has two isometric asymptotic ends, non-negative scalar curvature, and is geodesically complete.

In Step 3, along the lines of \cite{MuA}, we will conformally modify $(\widetilde{M}^{3},\widetilde{g}\,)$ into another geodesically complete, asymptotically flat Riemannian manifold $(\widehat{M}^{3}=\widetilde{M}^{3}\cup\lbrace \infty\rbrace,\widehat{g}\,)$.  By our choice of conformal factor, $(\widehat{M}^{3},\widehat{g}\,)$ is smooth and has non-negative scalar curvature away from the gluing surfaces and the point $\infty$, and suitably regular across them. The new manifold $(\widehat{M}^{3},\widehat{g}\,)$ will have precisely one end of vanishing ADM-mass.

In Step 4, by applying the rigidity statement of the positive mass theorem with suitably low regularity, we find that $(\widehat{M}^{3},\widehat{g})$ must be isometric to Euclidean space $(\R^{3},\delta)$. Thus, the original electrostatic system $\ES$ was conformally flat. It follows as in \cite{MuA} that it is indeed isometric to the exterior region $\lbrace r\geq \frac{3M}{2}+\frac{1}{2}\sqrt{9M^{2}-8Q^{2}}\rbrace$ of the photon sphere in the (spatial) \RN manifold of mass $M>0$ and charge $Q$ satisfying $Q^{2}>M^{2}$. This will complete our proof.

\subsection*{Step 1: Constructing an asymptotically flat manifold with minimal boundary and non-negative scalar curvature}
First, we recall that every connected component $\surf_{i}$ must be a topological sphere by Proposition \ref{prop}, Lemma \ref{untrapped}. Now fix $i\in\lbrace{1,\dots,I\rbrace}$. We define the (Komar style) \emph{charge}
\begin{align}\label{def:qi}
q_{i}&:=-\frac{1}{4\pi}\int_{\surf_{i}}\frac{\nu(\Phi)}{N}\,dA=-\frac{\vert\surf_{i}\vert_{\sigma_{i}}}{4\pi}\frac{\nu(\Phi)_{i}}{N_{i}}=-\frac{\nu(\Phi)_{i}r_{i}^{2}}{N_{i}}
\end{align}
which behaves like charge in non-relativistic electrostatics in view of \eqref{SMEelectrovac1}. This allows to rewrite \eqref{eq:constraint} as
\begin{align}\label{eq:constraintqi}
 \frac{4}{3}&=r_{i}^{2} H_{i}^{2}+\frac{4}{3}\left(\frac{q_i}{r_{i}}\right)^{2}.
\end{align}

Moreover, we define the \emph{mass} $\mu_{i}$ of $\surf_i$ by
\begin{align}\label{def:mui}
 \mu_i&:=\frac{r_i}{3}+\frac{2q_{i}^{2}}{3r_{i}}
 \end{align}
 and observe that $\mu_{i}>0$ and $\mu_{i}^{2}-q_{i}^{2}=
 (r_{i}^{2}-q_i^2)(r_i^2 - 4q_{i}^{2})/9r_{i}^{2}>0$, 
since \eqref{eq:constraintqi} and the sub-extremality condition, $r_{i}H_{i}>1$, imply $r_{i}^{2}>4q_{i}^{2}$. This allows us to define the interval
 \begin{align}\label{def:Ii}
  I_i:=\left[s_{i}:=\mu_i+\sqrt{\mu_{i}^{2}-q_{i}^{2}},r_i=\frac{3\mu_{i}}{2}+\frac{1}{2}\sqrt{9\mu_{i}^{2}-8q_{i}^{2}}\right]\subset\R.
\end{align}

To each boundary component $\surf_i$ of $\slice$, we now  glue in a cylinder of the form $I_i\times\surf_i$. We do this in such a way that the original photon sphere component $\surf_i\subset\slice$ corresponds to the level $\lbrace{r_i\rbrace}\times\surf_i$ of the cylinder $I_i\times\surf_i$ and will  continue to call this gluing surface $\surf_i$. The resulting manifold $\widetilde{M}^{3}$ has inner boundary
\begin{align}\label{boundary}
\mathfrak{B}&:=\bigcup_{i=1}^{I}\,\lbrace s_{i}\rbrace\times\surf_{i}.
\end{align}

In the following, we construct an electrostatic system $(\widetilde{M}^{3},\widetilde{g},\widetilde{N},\widetilde{\Phi})$ which is smooth and electro-vacuum away from the gluing surface $\surf$ and geodesically complete up to the boundary $\mathfrak{B}$. On $I_i\times\surf_i$, we first set
\begin{align}\label{def:g}
 \gamma_i&:=\frac{1}{f_i(r)^{2}}dr^2+\frac{r^{2}}{r_i^{2}}\sigma_i=\frac{1}{f_i(r)^{2}}dr^2+r^{2}\,\Omega,\\\label{def:fi}
 f_i(r)&:=\sqrt{1-\frac{2\mu_i}{r}+\frac{q_{i}^{2}}{r^{2}}},\\\label{def:phii}
 \varphi_{i}(r)&:=\frac{q_{i}}{r},
\end{align}
where $r\in I_i$ denotes the coordinate along the cylinder $I_i\times\surf_i$ and we have used that $\sigma_i=r_{i}^{2}\,\Omega$  by Proposition \ref{prop}, with $\Omega$ the canonical metric on $\mathbb{S}^{2}$. This is of course a portion of the spatial \RN system $(I_{i}\times\surf_{i},\gamma_i,f_{i},\varphi_{i})$ of mass $\mu_i>0$ and charge $q_{i}$ satisfying the sub-extremality condition $q_{i}^{2}<\mu_{i}^{2}$, namely the portion from the minimal surface to the photon sphere. It thus satisfies the electro-vacuum equations \eqref{SMEelectrovac1}-\eqref{SMEelectrovac3}. As \eqref{SMEelectrovac1}-\eqref{SMEelectrovac3} are invariant under $N\to \alpha N$, $\Phi\to\alpha\Phi+\beta$ for any constants $\alpha\neq0$, $\beta\in\R$, we can freely choose $\alpha_{i}\neq0$ and $\beta_{i}\in\R$ and obtain a new electrostatic electro-vacuum system $(I_{i}\times\surf_{i},\gamma_i,\alpha_{i}f_{i},\alpha_{i}\varphi_{i}+\beta_{i})$. 

It will be convenient to choose
\begin{align}\label{def:alphai}
\alpha_{i}&:=\frac{N_i}{f_i(r_{i})}>0,\\\label{def:betai}
\beta_{i}&:=\Phi_{i}-\alpha_{i}\frac{q_{i}}{r_{i}}.
\end{align}

Now combine $\ES$ with $(I_{i}\times\surf_{i},\gamma_i,\alpha_{i}f_{i},\alpha_{i}\varphi_{i}+\beta_{i})$ to obtain $(\widetilde{M}^{3},\widetilde{g},\widetilde{N},\widetilde{\Phi})$ by setting $\widetilde{g}:=g$ on $M^{3}$, $\widetilde{g}:=\gamma_i$ on $I_{i}\times\surf_{i}$, and
\begin{align}\label{def:tildeN}
\widetilde{N}:\widetilde{M}^3\to\R^{+}: p\mapsto
   \begin{cases}
     N(p) & \text{if } p\in \slice,\\[0.25cm]
     \alpha_{i}f_{i}(r(p))\phantom{+\beta_{i}ll} & \text{if } p\in I_i\times\surf_i,
         \end{cases}\\\label{def:tildePhi}
\widetilde{\Phi}:\widetilde{M}^3\to\R: p\mapsto
   \begin{cases}
     \Phi(p) & \text{if } p\in \slice,\\[0.25cm]
     \alpha_{i}\varphi_{i}(r(p))+\beta_{i} & \text{if } p\in I_i\times\surf_i,
         \end{cases}         
\end{align}
The metric $\widetilde{g}$ as well as the lapse $\widetilde{N}$ and the electric potential $\widetilde{\Phi}$ are naturally smooth away from the gluing surfaces $\surf_i$. The manifold $(\widetilde{M}^{3},\widetilde{g}\,)$ is geodesically complete up to the minimal boundary $\mathfrak{B}$ as $(\slice,g)$ was assumed to be geodesically complete up to its inner boundary. Moreover, it has non-negative scalar curvature away from the gluing surfaces as both $(\slice,g)$ and the rescaled \RN satisfy \eqref{SMEelectrovac1}-\eqref{SMEelectrovac3}. In a moment we will show that $\widetilde g$, $\widetilde{N}$, and $\widetilde{\Phi}$ are well defined and $C^{1,1}$ across $\surf$.\\

\noindent
\emph{It remains to show that $(\widetilde{M}^{3},\widetilde{g},\widetilde{N},\widetilde{\Phi})$ is $C^{1,1}$ across all gluing surfaces.}  To show this, we intend to use 
\begin{align}\label{def:psi}
\psi&:=\widetilde{N}:\widetilde{M}^3\to\R^+
\end{align}
as a smooth collar function across the gluing surfaces. Let us now fix $i\in\lbrace{1,\dots,I\rbrace}$.

Let us first show that $\psi$ is indeed well-defined and can be used as a smooth coordinate in a neighborhood of the gluing surface $\surf_{i}$: First, by construction, $\psi$ is smooth away from $\surf_{i}$. Then, by our choice of constant factor\footnote{In fact, this constant factor $\alpha_{i}$ can be expressed as $\alpha_{i}=\sqrt{3}m_{i}r_{i}/\sqrt{(r_{i}^{2}-q_{i}^{2})(r_{i}^{2}-2\mu_{i}r_{i}+q_{i}^{2}})$ with $m_i$ the Komar mass of $\Sigma_i$, $m_i=\frac{1}{4\pi}\int_{\surf_i}\nu(N)\,dA=\frac{\lvert\surf_i\rvert_{\sigma_i}}{4\pi}\nu(N)_{i}=r_{i}^{2}\,\nu(N)_{i}$. Note that $m_{i}>0$ by \eqref{eq:prop'} and Lemma \ref{untrapped}. If $q_i=0$, this collapses to $\alpha_{i}=3m_{i}/r_{i}$ which coincides with the multiplicative factor chosen in \cite{cedergal}.} $\alpha_i$ in front of $f_i$ (which equals $1$ in case $(\slice,g)$ already is a \RN manifold), $\psi$ has the same constant value on each side of $\surf_i$, and hence is well-defined and continuous across $\surf_i$.

The outward unit normal vector to $\surf_i$ with respect to the \RN side is given by $\widetilde{\nu}=f_{i}(r_{i})\,\partial_{r}$.  In $\slice$, the outward unit normal is given by $\widetilde{\nu}=\nu$. By our choice of $\mu_i$ and $q_{i}$, one verifies, using Proposition \ref{prop}, Lemma \ref{untrapped}, \eqref{def:qi}, and \eqref{def:mui}, that the normal derivative of $\psi$ is the same {\it positive} constant on both sides of $\surf_i$. This, in particular, allows us to use $\psi$ as a smooth coordinate function in a (collared) neighborhood of each $\surf_i$ in $\widetilde{M}^{3}$. As $\widetilde{\nu}(\psi)\vert_{\surf_i}$ coincides from both sides, the normal $\widetilde{\nu}$ is in fact continuous and thus by smoothness even $C^{0,1}$ across $\surf_i$. Moreover, this shows that the lapse function $\widetilde{N}=\psi$ is $C^{1,1}$ across $\surf_i$.

By choice of $\beta_{i}$, $\widetilde{\Phi}$ is also continuous across $\surf_{i}$. Its normal derivative on the \RN side satisfies $\widetilde{\nu}(\widetilde{\Phi})=\alpha_{i}\widetilde{\nu}(\varphi_{i})=\alpha_{i}(-q_{i}f_{i}/r_{i}^{2})=-q_{i}\psi/r_{i}^{2}$, where we have used the explicit form of Reissner-Nordstr\"om. By definition of $q_{i}$, \eqref{def:qi}, the same identity $\widetilde{\nu}(\widetilde{\Phi})=-q_{i}N_{i}/r_{i}^{2}=-q_{i}\psi/r_{i}^{2}$ also holds on the original side of $\surf_{i}$. As $\widetilde{\Phi}$ is smooth in a deleted neighborhood of $\surf_{i}$, this shows that $\widetilde{\Phi}$ is indeed $C^{1,1}$ across $\surf_{i}$.

To show that $\widetilde{g}$ is $C^{1,1}$ across $\surf_i$, let $(y^{A})$ be local coordinates on $\surf_i$ and flow them to a neighborhood of $\surf_i$ in $\widetilde{M}^{3}$ along the level set flow defined by $\psi$. It then suffices to show that the components $\widetilde{g}_{AB}$, $\widetilde{g}_{A\psi}$, and $\widetilde{g}_{\psi\psi}$ are $C^{1,1}$ with respect to the coordinates $(y^{A},\psi)$ across the $\psi$-level set $\surf_i$ for all $A,B=1,2$. 

Continuity of $\widetilde g$ in the coordinates $(y^{A},\psi)$ and smoothness in tangential directions along $\surf_i$ is then immediate as $\partial_{\psi}=\frac{1}{\widetilde{\nu}(\psi)}\,\widetilde{\nu}$, and thus
\begin{align}\label{continuous}
\widetilde{g}_{AB}&=r_{i}^{2}\,\Omega_{AB},\quad \widetilde{g}_{A\psi}=0, \quad \widetilde{g}_{\psi\psi}=\frac{1}{(\widetilde{\nu}(\psi))^{2}}
\end{align}
on $\surf_{i}$ for all $A,B=1,2$ and from both sides. Further, we find that
\begin{align*}
\partial_{\psi}\left(\widetilde{g}_{AB}\right)&=\frac{2}{\widetilde{\nu}(\psi)}\,\widetilde{h}_{AB}
\end{align*}
holds on $\surf_{i}$, where $\widetilde{h}_{AB}$ is the second fundamental form induced on $\surf_i$ by $\widetilde{g}$. 
Proposition \ref{prop} ensures the umbilicity of every component of any photon sphere. Also, it asserts that the mean curvature of every component of any photon sphere is determined by its area radius $r_{i}$ and charge $q_{i}$ via \eqref{eq:constraintqi}, up to a sign. Hence, by using \eqref{continuous}, we observe that 
$\widetilde{h}=\pm\frac{1}{2}H_{i}\sigma_{i}=\pm\frac{1}{2}H_{i}\,r_{i}^{2}\,\Omega$ holds on both sides of the photon sphere gluing boundary component $\surf_i$. We still need to determine this a priori free sign: From the side of $\slice$, we know from Lemma \ref{untrapped} that $H_{i}>0$, where $H_{i}$ is computed with respect to the unit normal pointing towards the asymptotic end. On the \RN side, the mean curvature of the photon sphere with respect to the unit normal $\widetilde{\nu}$ pointing towards infinity and thus into $\slice$, is also positive. Thus, in both cases $\widetilde{h}_{AB}$ and thus also $\partial_{\psi}\left(\widetilde{g}_{AB}\right)$ coincide from both sides of $\surf_i$ for all $A,B=1,2$.

By construction, $\widetilde{g}_{A\psi}=0$ not only on $\surf_i$ but also in a neighborhood of $\surf_i$ inside $\widetilde{M}^{3}$ so that $\partial_{\psi}(\widetilde{g}_{A\psi})=0$ on both sides of $\surf_i$ for $A=1,2$. It remains to be shown that $\partial_{\psi}(\widetilde{g}_{\psi\psi})$ coincides from both sides. Then, because $\widetilde{g}$ is smooth on both sides up to the boundary (by assumption), we will have proved that $\widetilde{g}$ is $C^{1,1}$ everywhere.

A direct computation using the level set flow equations for $\psi$ shows that 
\begin{align}
\partial_{\psi}(\widetilde{g}_{\psi\psi})&=-2\,(\widetilde{\nu}(\psi))^{6}\,\,\my{\widetilde{g}}{\nabla}^{2}\psi(\widetilde{\nu},\widetilde{\nu})
\end{align}
from both sides of $\surf_i$, where $\my{\widetilde{g}}{\nabla}^{2}\psi$ denotes the Hessian of $\psi$ with respect to $\widetilde{g}$. We already know that $\widetilde{\nu}(\psi)$ is continuous across $\surf_i$. From the general identity for the splitting of the Laplacian on hypersurfaces, we know that
\begin{align}
\my{\widetilde{g}}{\nabla}^{2}\psi(\widetilde{\nu},\widetilde{\nu})\stackrel{\eqref{def:psi}}{=}\my{\widetilde{g}}{\nabla}^{2}\widetilde{N}(\widetilde{\nu},\widetilde{\nu})=\mylap{\widetilde{g}}\widetilde{N}-\mylap{\widetilde{\sigma_{i}}}\widetilde{N}-\widetilde{H}_{i}\,\widetilde{\nu}(\widetilde{N}),
\end{align}
where $\widetilde{H}_{i}$ denotes the mean curvature induced by $\widetilde{g}$ with respect to $\widetilde{\nu}$, and $\mylap{\widetilde{g}}$ and $\mylap{\widetilde{\sigma_{i}}}$ denote the $3$- and $2$-dimensional Laplacians induced by $\widetilde{g}$ and $\widetilde{\sigma_{i}}:=\widetilde{g}\vert_{T\surf_{i}\times T\surf_{i}}$, respectively. Using that $\psi$ is constant along $\surf_i$ and \eqref{SMEelectrovac2}, it follows that
\begin{align}\label{Hessian}
\my{\widetilde{g}}{\nabla}^{2}\psi(\widetilde{\nu},\widetilde{\nu})&=\cancelto{\frac{\vert d\widetilde{\Phi}\vert^{2}_{\widetilde{\sigma}_{i}}}{\widetilde{N}}}{\mylap{\widetilde{g}}\widetilde{N}}\;\;-\;\;\cancelto{0}{\mylap{\widetilde{\sigma_{i}\!}}\widetilde{N}}\;-\;\widetilde{H}_{i}\,\widetilde{\nu}(\widetilde{N})=\frac{\widetilde{\nu}(\widetilde{\Phi})^{2}}{\widetilde{N}}-\widetilde{H}_{i}\,\widetilde{\nu}(\widetilde{N}),
\end{align}
on both sides of $\surf_i$. To conclude, we recall that $\widetilde{\nu}(\widetilde{\Phi})$, $\widetilde{N}$, $\widetilde{H}_{i}$, and $\widetilde{\nu}(\widetilde{N})$ are continuous across $\surf_i$ so that $\my{\widetilde{g}}{\nabla}^{2}\psi(\widetilde{\nu},\widetilde{\nu})$ and thus $\partial_{\psi}(\widetilde{g}_{\psi\psi})$ are continuous across $\surf_i$. Thus, $\widetilde{g}$ is $C^{1,1}$ across $\surf_i$. As $i\in\lbrace1,\dots,I\rbrace$ was arbitrary, $(\widetilde{M}^{3},\widetilde{g},\widetilde{N},\widetilde{\Phi})$ is indeed $C^{1,1}$. 

\subsection*{Step 2: Doubling}
Now, we rename $\widetilde{M}^{3}$ to $\widetilde{M}^{+}$, reflect $\widetilde{M}^{+}$ through $\mathfrak{B}$ to obtain $\widetilde{M}^{-}$, and  glue the two copies to each other along their shared minimal boundary~$\mathfrak{B}$. We thus obtain a new smooth manifold which we will call $\widetilde{M}^3$ by a slight abuse of notation. We define a metric on $\widetilde{M}^{3}$ by equipping both $\widetilde{M}^{\pm}$ with the metric $\widetilde{g}$ constructed in Step 1. Respecting the symmetries of the electro-vacuum equations \eqref{SMEelectrovac1}-\eqref{SMEelectrovac3}, we extend the lapse function $\widetilde{N}^{+}:=\widetilde{N}$ and electric potential $\widetilde{\Phi}^{+}:=\widetilde{\Phi}$ constructed in Step 1 on $\widetilde{M}^{+}$ across $\mathfrak{B}$ by choosing $\widetilde{N}^{-}:=-\widetilde{N}^{+}$ (odd) and $\widetilde{\Phi}^{-}:=\widetilde{\Phi}$ (even). Combined, we will call the extended functions $\widetilde{N}:=\pm\widetilde{N}^{+}$, $\widetilde{\Phi}:=\widetilde{\Phi}^{\pm}$ on $\
widetilde{M}^{\pm}$.\\

\noindent \emph{Why $(\widetilde{M}^{3},\widetilde{g},\widetilde{N},\widetilde{\Phi})$ is smooth across $\mathfrak{B}$ and $\widetilde{\psi}:=\widetilde{N}$ can be used as a smooth collar coordinate function near $\mathfrak{B}$.}  In contrast to \cite{MuA}, this is in fact almost immediate because we are just gluing two \RN necks of the same mass $\mu_{i}$ and charge $q_{i}$ to each other across their minimal surface boundaries on each component $\lbrace s_{i}\rbrace\times\surf_{i}$ of $\mathfrak{B}$---up to a constant/affine rescaling of the lapse functions and electric potentials, see  \eqref{def:tildeN}, \eqref{def:tildePhi}. This rescaling respects the odd symmetry of $f_{i}$ and the even symmetry of $\varphi_{i})$ in \RN and thus does not cause any additional problems. Indeed, $\widetilde{\psi}$ and $\widetilde{\Phi}$ are smooth across the horizon boundary $\mathfrak{B}$ as can be seen in isotropic coordinates. Smoothness of the metric $\widetilde{g}$ across $\mathfrak{B}$ then follows directly.\\

By construction, the doubled electrostatic system $(\widetilde{M}^{3},\widetilde{g},\widetilde{N},\widetilde{\Phi})$ has two isometric ends which are asymptotic to \RN of mass $M$ and charge $Q$. It is geodesically complete as $(\slice,g)$ was assumed to be geodesically complete up to the boundary. Finally, we observe that it satisfies the electro-vacuum equations away from $(\surf_i)^{\pm}$ by construction and that it is smooth away from and $C^{1,1}$ across~$(\surf_i)^{\pm}$.

\subsection*{Step 3: Conformal transformation to a geodesically complete manifold with vanishing ADM-mass and non-negative scalar curvature}
As in Masood-ul-Alam \cite{MuA}, we want to use 
\begin{align}\label{def:Omega}
{\Omega}&:=\frac{1}{4}\left(\left(1+\widetilde{N}\right)^{2}-\widetilde{\Phi}^{2}\right)
\end{align}
as a conformal factor on $\widetilde{M}^{3}$. However, in our situation it is not a priori evident that~${\Omega}>0$.\\

\paragraph*{\emph{Why ${\Omega}>0$ holds on $\widetilde{M}^{3}$}.}
By the odd-even-symmetric definition of $\widetilde{N}^{\pm}=\pm\widetilde{N}$ and $\widetilde{\Phi}^{\pm}=\widetilde{\Phi}$ in Step 2, it suffices to show that $(1\pm\widetilde{N})^{2}>\widetilde{\Phi}^{2}$ in $\widetilde{M}^{+}$. Clearly, as $\widetilde{N}>0$ in $\widetilde{M}^{+}$, we have $(1+\widetilde{N})^{2}>(1-\widetilde{N})^{2}$ so that is suffices to show $(1-\widetilde{N})^{2}>\widetilde{\Phi}^{2}$ in $\widetilde{M}^{+}$. To do so, we re-write $(1-\widetilde{N})^{2}-\widetilde{\Phi}^{2}=(\widetilde{N}-1+\widetilde{\Phi})(\widetilde{N}-1-\widetilde{\Phi})$ in $\widetilde{M}^{+}$. Now, observe that
\begin{align}\label{eq:elliptic}
\mylap{\widetilde{g}}(\widetilde{N}-1\pm\widetilde{\Phi})\mp\frac{d\widetilde{\Phi}(\,\my{\widetilde{g}}\grad(\widetilde{N}-1\pm\widetilde{\Phi}))}{\widetilde{N}}&=0
\end{align}
as a consequence of \eqref{SMEelectrovac1} and \eqref{SMEelectrovac2}. This is an elliptic PDE to which the maximum principle applies as $\widetilde{N}>0$ (see e.\,g.\ \cite{GT}). If we can show that $\widetilde{N}-1\pm\widetilde{\Phi}<0$ on $\widetilde{M}^{+}$, it follows that $\Omega>0$ on $\widetilde{M}^{3}$. 

On $M^{+}=\slice$, $\widetilde{N}-1\pm\widetilde{\Phi}=N-1\pm\Phi$ and we know that $N\to1$ and $\Phi\to0$ as $r\to\infty$ so that $\widetilde{N}-1\pm\widetilde{\Phi}\to0$ as $r\to\infty$. On the other hand, we find that
\begin{align}\label{eq:sign}
\widetilde{\nu}(\widetilde{N}-1-\widetilde{\Phi})\vert_{\surf_{i}}&\;\;=\phantom{\frac{1}{2}}\nu(N)_{i}\,\pm\nu(\Phi)_{i}\stackrel{\eqref{eq:prop'}}{=}\frac{1}{2}H_{i}N_{i}\pm\nu(\Phi)_{i}
\end{align}
and thus $\widetilde{\nu}(\widetilde{N}-1-\widetilde{\Phi})\vert_{\surf_{i}}>0$ if and only if
\begin{align}\label{eq:sign2}
&\left(\frac{1}{2}H_iN_i\right)^2>\left(\nu(\Phi)_i\right)^2\quad\stackrel{\eqref{def:qi}}{\Leftrightarrow}\quad H_i^2r_i^2 >4\frac{q_i^2}{r_i^2} \\\nonumber
\stackrel{\eqref{eq:constraintqi}}{\Leftrightarrow}&\quad \frac{4}{3}\left(1-\frac{q_i^2}{r_i^2}\right)>4\frac{q_i^2}{r_i^2}\quad\quad \Leftrightarrow \quad\frac{q_i^2}{r_i^2}>\frac{1}{4}\\\nonumber
\stackrel{\eqref{eq:constraintqi}}{\Leftrightarrow}&\quad H_i^2r_i^2>1.
\end{align}
Since the latter holds by the sub-extremality condition, the maximum principle implies that 
$\widetilde{N}-1\pm\widetilde{\Phi}<0$ on the original manifold $M^{+}=M^{3}$. 

On each neck $(I_{i}\times\surf_{i})^{+}=(I_{i}\times\surf_{i})$,
$\widetilde{N}=\alpha_{i}f_i$, where $f_{i}$ is the \RN lapse function given by \eqref{def:fi} and $\alpha_{i}>0$. It is easy to compute
\begin{align}
\widetilde{\nu}(\widetilde{N}-1\pm\widetilde{\Phi})&=\widetilde{\nu}(\alpha_{i}f_{i}-1\pm\left(\alpha_{i}\varphi_{i}+\beta_{i})\right)=\alpha_{i}f_{i}\left(f_{i},_{r}\pm\varphi_{i},_{r}\right)>0
\end{align}
explicitly in Reissner-Nordstr\"om. Again by the maximum principle applied to \eqref{eq:elliptic}---or by direct computation in Reissner-Nordstr\"om---it follows that $\widetilde{N}-1\pm\widetilde{\Phi}<0$ everywhere on $\widetilde{M}^{3}$ and thus $\Omega>0$ on $\widetilde{M}^{3}$.\\

The above considerations allow us to define the conformally transformed metric
\begin{align}\label{def:ghat}
\widehat{g}&:=\Omega^4\,\widetilde{g}
\end{align}
on $\widetilde{M}^3$. Away from the gluing surfaces, $\Omega$ is smooth 
as $\widetilde{N}$ and $\widetilde{\Phi}$ are smooth, there. Since, in addition, $\widetilde{g}$ is smooth away from the gluing surfaces, the same holds true for $\widehat{g}$. Following the exposition on p.~153ff in Heusler \cite{Heusler}, we find
\begin{align}
\frac{\Omega^{4}}{2}\,\my{\widehat{g}\!}{R}&=\frac{1}{(4\widetilde{N})^{2}}\left\lvert(1-\widetilde{N}^{2}-\widetilde{\Phi}^{2})\,d\widetilde{\Phi}+2\,\widetilde{\Phi}\,\widetilde{N}\,d\widetilde{N}\right\rvert^{2}_{\widetilde{g}}\geq0,
\end{align}
where $\my{\widehat{g}\!}{R}$ denotes the scalar curvature of $\widetilde{M}^3$ with respect to $\widehat{g}$.

Furthermore, $\Omega$ and $\widehat{g}$ are $C^{1,1}$ across all gluing boundaries $\surf_i$, $i\in\lbrace1,\dots,I\rbrace$, because $\widetilde{N}$, $\widetilde{\Phi}$, and $\widetilde{g}$ are $C^{1,1}$ there (by product rule and \eqref{def:Omega}, \eqref{def:ghat}).

Moreover, precisely as in \cite{MuA}, $(M^+,\widehat{g}\,)$ is asymptotically flat with zero ADM-mass. On the other hand, again precisely as in \cite{MuA}, $(M^-,\widehat{g}\,)$ can be compactified by adding in a point $\infty$ at infinity with $(\widehat{M}^{3}:=\widetilde{M}^3\cup\lbrace \infty\rbrace,\widehat{g}\,)$ sufficiently regular. By construction, $(\widehat{M}^{3},\widehat{g}\,)$ is geodesically complete.

Summarizing, we now have constructed a geodesically complete Riemannian manifold $(\widehat{M}^{3},\widehat{g}\,)$ with non-negative scalar curvature $\my{\widehat{g}\!}{R}\geq0$ and one asymptotically flat end of vanishing ADM mass that is smooth away from some hypersurfaces and one point. {At the point $\infty$, as well as at all gluing surfaces, the regularity is precisely that encountered by \cite{MuA}.}

\subsection*{Step 4: Applying the Positive Mass Theorem.}
In Steps 1-3, we have constructed the geodesically complete Riemannian manifold $(\widehat{M}^{3},\widehat{g}\,)$ of non-negative scalar curvature which has one asymptotically flat end of vanishing ADM mass in a manner similar to what is done in \cite{MuA}. Moreover, as noted above, the regularity achieved is the same as that encountered in \cite{MuA}. Masood-ul-Alam's analysis thus fully applies and asserts that the (weak regularity) Positive Mass Theorem proved by Bartnik \cite{Bartnik} applies.

The rigidity statement of this (weak regularity) Positive Mass Theorem implies that $(\widehat{M}^{3},\widehat{g}\,)$ must be isometric to Euclidean space $(\R^3,\delta)$. This immediately shows that the photon sphere $\photo$ was connected and diffeomorphic to a cylinder over a sphere for topological reasons. Moreover, it allows us to deduce that $(\slice,g)$ must be conformally flat. It is well-known\footnote{and can be verified by a  computation using the Bach tensor, see e.\,g.\ \cite{Heusler} or \cite{MuA}.} that the only conformally flat, maximally extended solution of the electro-vacuum equations \eqref{SMEelectrovac1}-\eqref{SMEelectrovac3} is the \RN solution \eqref{RNmetric}.

In particular, the lapse function $N$ is given by $N=\sqrt{1-2M/r+Q^{2}/r^{2}}$ with $r$ the area radius along the level sets of $N$, and $M$ and $Q$ the mass and charge of $\ES$. Equation \eqref{SMEelectrovac1} then immediately implies $Q=q_{1}$ by the divergence theorem. This, in turn, gives $\mu_{1}=M$ via \eqref{def:mui} so that in particular $M>0$ and $Q^{2}<M^{2}$. This proves the claim of Theorem \eqref{thm:main}.
\end{proof}

\section{The electrostatic $n$-body problem for very compact bodies and black holes}\label{sec:nbody}
The following theorem addresses the ``electrostatic $n$-body problem'' in General Relativity, namely the question whether multiple suitably ``separated'' (potentially) electrically charged bodies and black holes can be in static equilibrium. This problem has received very little attention; most progress has been achieved in the non-electrically charged case, see \cite{cedergal} and references cited therein. The most notable result in the charged realm is Masood-ul-Alam's theorem \cite{MuA} on electrostatic black hole uniqueness. This can be re-interpreted as saying that there are no $n>1$ (potentially) electrically charged non-degenerate black holes in static equilibrium. Our approach can handle extended bodies as well as combinations of bodies and black holes and makes no symmetry assumptions. However, we can only treat ``very compact'' bodies, namely bodies that are each so compact that they give rise to a (sub-extremal) photon sphere behind which they reside:

\begin{theorem}[No static configuration of $n$ ``very compact'' electrically charged bodies and black holes]\label{thm:nbody}
There are no electrostatic equilibrium configurations of $k\in~\!\!\N$ possibly electrically charged bodies and $n\in\N$ possibly electrically charged non-degenerate black holes with $k+n>1$ in which each body is surrounded by its own sub-extremal photon sphere.
\end{theorem}

To be specific, the term \emph{static equilibrium} is interpreted here as referring to an electrostatic electro-vacuum system $(\overline{M}^{3},\overline{g},\overline{N},\overline{\Phi})$, asymptotic to \RN and geodesically complete up to (possibly) an inner boundary consisting of one or multiple \emph{black holes}, defined as sections of a Killing horizon (or in other words consisting of totally geodesic topological spheres satisfying $\overline{N}=0$). Furthermore, a \emph{body} is meant to be a bounded domain $\Omega\subset\overline{M}^{3}$ where the (electrostatic) Einstein equations hold with a right hand side coming from the energy momentum tensor of a matter model satisfying the dominant energy condition $\my{\overline{g}}{\Scal}\geq0$. We consider a body $\Omega$ to be \emph{very compact} if it creates a sub-extremal photon sphere $\surf$ according to Definition \ref{def:sub-extremal} that, without loss of generality, arises as its boundary, $\surf=\partial\Omega$. Naturally, all bodies are 
implicitly assumed to be disjoint. Outside all bodies, the system is assumed to satisfy the electrovacuum 
equations \eqref{SMEelectrovac1}-\eqref{SMEelectrovac3}.

Theorem \ref{thm:main} then applies to the spacetime $(\mathfrak{L}^{4}=\spacet,\mathfrak{g}=-\overline{N}^{2}dt^{2}+\overline{g}\vert_{\slice})$, where $\slice:=\overline{M}^{3}\setminus\cup_{i=1}^{I}\Omega_{i}$, and $\emptyset\neq\Omega_{i}\subset\overline{M}^{3}$, $i\in\lbrace1,\dots,I\rbrace$, are all the bodies in the system. We appeal to Remark \ref{rem:multiple} if black holes are present in the configuration.\qed

\smallskip
Indeed, while photon spheres might be most well known from the electro-vacuum \RN spacetime \eqref{RNmetric}, many astrophysical objects are believed to be surrounded by photon spheres, see e.\,g.\ \cite{LY} and references cited therein.
\newpage

\appendix\section{Sub-extremality}\label{app:sub-extremal}
Theorem \ref{thm:main} only applies to what we call \emph{sub-extremal} photon spheres in Definition \ref{def:sub-extremal}. This sub-extremality condition is used twice in the proof of Theorem \ref{thm:main}: First, it is used to ensure that the $(\mu_i,q_i)$-\RN neck with $\mu_i$ and $q_i$ defined as in \eqref{def:mui} and \eqref{def:qi}, respectively, belongs to a sub-extremal \RN spatial slice and thus possess a non-degenerate horizon which allows for doubling. Sub-extremality of $(\mu_i,q_i)$-\RN---in the sense that $q_i^2<\mu_i^2$---can indeed be seen to be equivalent to $\Sigma_i$ being sub-extremal in the sense that $H_ir_i>1$ for each $i=1,\dots,I$. Second, we used the sub-extremality condition, $H_ir_i>1$ for all $i=1,\dots,I$,  to ensure that the conformal factor $\Omega$ introduced in Step 4 stays strictly positive, see Step 3. Again, the condition $\Omega>0$ can be seen to be equivalent to $H_ir_i>1$ for all $i=1,\dots,I$ via \eqref{eq:constraintqi} by \eqref{eq:sign}, \eqref{eq:sign2}.

In the following, we will argue that the terminology introduced in Definition \ref{def:sub-extremal} is justified beyond its usefulness in the proof of Theorem \ref{thm:main}. This justification only applies to systems that possess a connected photon sphere boundary (and no additional black hole boundary components) and extends arguments from \cite{LY}. It is well-known that the lapse function $N$ and the electric potential $\Phi$ of an electrostatic, electro-vacuum 
system $(M^3,g,N,\Phi)$ asymptotic to Reissner-Nordstr\"om of total mass $M$ and charge $Q$, with suitable inner boundary (e.\,g.\ (Killing horizon) \cite{Heusler} or photon sphere \cite{LY}), satisfies the functional relationship  
\begin{align}\label{eq:functional}
 N^2&=1+\Phi^2-2\frac{M}{Q}\Phi
\end{align}
unless the total charge $Q$ vanishes (in which case it is clearly justified to call the photon sphere inner boundary sub-extremal). The \emph{Komar mass} $m$ of the inner boundary $\surf=\partial M^3$ is given by
\begin{align}\label{def:m}
m&:=\frac{1}{4\pi}\int_{\surf}\nu(N)\,dA=\frac{\lvert\surf\rvert_{\sigma}}{4\pi}\nu(N)=r^{2}\,\nu(N),
\end{align}
where $dA$ denotes the area measure with respect to $\sigma$ and $r$ denotes the area radius, see \eqref{eq:arearadius}. By the divergence theorem, one then obtains
\begin{align*}
 M&\stackrel{\eqref{def:m}}{=}m+\frac{1}{4\pi}\int_{M^3}\frac{\vert d\Phi\vert^2}{N}\,dV\\
 &\;=\;m-\frac{1}{4\pi}\left(\int_{M^3}\Phi\underbrace{\diver\left(\frac{\grad\Phi}{N}\right)}_{=0\text{ by \eqref{SMEelectrovac1}}}\,dV-\underbrace{\lim_{r\to\infty}\int_{\mathbb{S}^2_r}\Phi\frac{\nu(\Phi)}{N}\,dA}_{=0\text{ by \eqref{AF}, \eqref{NAF}, \eqref{PhiAF}}}+\int_\Sigma\Phi\frac{\nu(\Phi)}{N}\,dA\right)\\
 &\stackrel{\eqref{def:qi}}{=}m+\Phi_0Q,
\end{align*}
where $dV$ denotes the volume measure on $M^3$ induced by $g$, and $\Phi_0:=\Phi\vert_\Sigma$ is constant by assumption.

Plugging this into \eqref{eq:functional} as $\Phi_0=(M-m)/Q$, we find 
\begin{align}\label{eq:N0}
N_0^2=\frac{Q^2+m^2-M^2}{Q^2},
\end{align}
where $N_0:=N\vert_\Sigma$ is constant, and thus
\begin{align}
H^2\,\frac{Q^2+m^2-M^2}{Q^2}&\stackrel{\eqref{eq:N0}}{=}H^2 N_0^2\stackrel{\eqref{eq:prop'}}{=}4\left(\nu(N)\right)^2\stackrel{\eqref{def:m}}{=}4\frac{m^2}{r^4}  \,.
\end{align}
Using \eqref{eq:constraintqi}, this is equivalent to,
\begin{align}
\frac{1}{H^2r^2}& =1+\frac{Q^2-M^2}{4m^2}
\end{align}
so that $Q^2<M^2$ if and only if $Hr>1$ (sub-extremal case), $Q^2=M^2$ if and only if $Hr=1$ (extremal case), and $Q^2>M^2$ if and only if $Hr<1$ (super-extremal case) holds a priori---at least for a (connected) photon sphere in electro-vacuum, see also Lemma \ref{untrapped}. We end by remarking that the Majumdar-Papapetrou spacetimes with more than one mass \cite[p.~161ff]{Heusler} do not possess any photon spheres as can be verified by direct computation using Definition \ref{def:photo}.

\vfill
\bibliographystyle{amsplain}
\bibliography{photon-sphere}
\vfill
\end{document}